\documentclass{amsart}
\usepackage{eucal}
\usepackage{color}
\usepackage{verbatim}
\usepackage{graphicx}

\usepackage{hyperref}

\theoremstyle{plain}
\newtheorem{theorem}{Theorem}[section]
\newtheorem{lemma}[theorem]{Lemma}
\newtheorem{proposition}[theorem]{Proposition}
\newtheorem{corollary}[theorem]{Corollary}
\theoremstyle{definition}
\newtheorem{definition}[theorem]{Definition}

\theoremstyle{definition}
\theoremstyle{remark}
\newtheorem{remark}{Remark}

\renewcommand{\d}{\textup{d\,}}

\makeatletter
\def\@seccntformat#1{\csname the#1\endcsname.\quad}
\def\section{\@startsection{section}{1}{\z@}%
{-3.5ex \@plus -1ex \@minus -.2ex}%
{2.3ex \@plus.2ex}%
{\Large\bf}}
\def\subsection{\@startsection{subsection}{2}{\z@}%
{-3.25ex\@plus -1ex \@minus -.2ex}%
{1.5ex \@plus .2ex}%
{\bf\large}}
\def\subsubsection{\@startsection{subsubsection}{3}{\z@}%
{-3.25ex\@plus -1ex \@minus -.2ex}%
{1.5ex \@plus .2ex}%
{\normalfont\normalsize\bf}}

\def\blfootnote{\xdef\@thefnmark{}\@footnotetext}
\makeatother

\newcommand{\N}{\mathbb{N}}
\newcommand{\R}{\mathbb{R}}

\newcommand{\e}{\varepsilon}
\newcommand{\dd}{\partial}
 \newcommand{\be}{\begin{equation}}
 \newcommand{\ee}{\end{equation}}
 \newcommand{\bd}{\begin{displaymath}}
 \newcommand{\ed}{\end{displaymath}}
 \newcommand{\bea}{\begin{eqnarray}}
 \newcommand{\eea}{\end{eqnarray}}
 \newcommand{\beas}{\begin{eqnarray*}}
 \newcommand{\eeas}{\end{eqnarray*}}
 \newcommand{\bc}{\begin{center}}
 \newcommand{\ec}{\end{center}}

 \newcommand{\pa}{\partial}

 \newcommand{\lap}{\Delta}

 \def \O{\Omega}

\title[Tangential touch in a semilinear free boundary problem in 2D]{Tangential touch between the free and the fixed boundary in a semilinear free boundary problem in two dimensions}

\date{\today}
\author{Mahmoudreza Bazarganzadeh}
\address{Department of Mathematics, Uppsala  University,
Box 480, 751 06 Uppsala, Sweden}
\email{reza@math.uu.se}
\thanks{}

\author{Erik Lindgren}
\address{Department of Mathematical Sciences, Norwegian university of science and technology, Sentralbygg 2, Alfred Getz vei 1, 7491 Trondheim, Norway}
\email{erik.lindgren@math.ntnu.no}
\thanks{}
\begin{document}


\maketitle

\begin{abstract}
We study minimizers of the functional
\beas
\int_{B_1^+}(|\nabla u|^2+2(\lambda^+(u^+)^p+\lambda^-(u^-)^p))\d x,
\eeas
where $B_1^+=B_1\cap \{x: x_1>0\}$, $\,\lambda^{\pm}$ are two positive constants and $0<p<1$. In two dimensions, we prove that the free boundary is a uniform $C^1$ graph up to the flat part of the fixed boundary. Here, the free boundary refers to the boundaries of the sets $\{\pm u>0\}$.

\end{abstract}

\section{Introduction}
\renewcommand{\theequation}{1.\arabic{equation}}\setcounter{equation}{0}

\subsection{Problem setting}
Let $u^\pm = \max\{\pm u,0\}$, $\Pi=\{x\in \R^n: x_1=0 \}$ and consider minimizers of the functional
\be\label{eq:minimizer}
 E(u)=\int_{B_1^+} \big(|\nabla u|^2+2\lambda^+(u^+)^{p}+2\lambda^-(u^-)^{p}\big)\d x,
\ee
over
 $$\mathcal{K}=\{u\in W^{1,2}(B_1^+)\,:\,u=0\,\,\text{on }{B_1}\cap \Pi\text{ and }  u=f\text{ on } \pa B_1^+\setminus \Pi\}.$$ Here $B_1$ is  the unit ball in $\R^n$ and $$f\in W^{1,2}(B_1) \cap L^\infty(B_1),\,\, \lambda^\pm>0,$$
$$B_1^+ = B_1\cap \{x: x_1>0\},\,\,0<p<1.$$
By classical methods in calculus of variations it is straight forward to prove the existence of a minimizer.
 The corresponding Euler Lagrange formulation of (\ref{eq:minimizer}) reads
\begin{align}\label{eq:problem}
\left\{\begin{array}{lr}
\lap u  = p\big(\lambda^+(u^+)^{p-1}\chi_{\{u>0\}}-\lambda^-(u^-)^{p-1}\chi_{\{u<0\}}\big)& \text{ in } B_1^+\cap \{u\neq 0\},\\
u=f&\text{ on } \pa B_1^+\setminus \Pi,\\
u=0&\text{ on }B_1\cap \Pi.\end{array}\right.
\end{align}
Due to the singularity of the Euler Lagrange equation, it is not clear that any minimizer satisfies the equation everywhere. Moreover, since the energy is not convex, there might be more than one minimizer with given boundary data.

We use the notation $\O^+=\{u>0\}$, $\O^-=\{u<0\}$, $\Gamma^\pm=\pa \O^\pm $, $\Gamma=\Gamma^+\cup\Gamma^-$ and refer to $\Gamma $ as the \textit{free boundary}
 which is not known a priori, i.e, it is a part of the solution of the problem.


The main result of this paper concerns the behavior of the free boundary close to the fixed boundary $\Pi$, in two dimensions. In order to state our main theorem, we define the class of solutions within which we will work.

\begin{definition} Let $M,R$ be two positive constants. We define $P_R(M)$ to be the class of minimizers $u$ of
 (\ref{eq:minimizer}) in $B_R^+$ such that $0\in \Gamma\cap \Pi$ and
  \beas
\|u\|_{L^{\infty}(B_R^+)}\leq M.
\eeas
 \end{definition}
 \begin{remark}
If $u$ is a minimizer of (\ref{eq:minimizer}) in $B_1^+$ such that
  \beas
\|u\|_{L^{\infty}(B_1^+)}\leq M,
\eeas
and $x_0\in \Gamma\cap \Pi$ but $0\not\in \Gamma\cap \Pi$, one can by translating and rescaling $u$ obtain a function in  $P_1(M')$, for another constant $M'$.
\end{remark}
\begin{theorem}\label{thm:main}
 Let $u\in P_1(M)$ in dimension two. Then, in a neighborhood of the origin, the free boundary is a $C^1$ graph with a modulus of continuity depending only on $M$.
 \end{theorem}



\subsection{Known Result}
 The one-phase case of the problem, i.e, the case when $u$ does not change signs, has been well studied before. Phillips has proved in \cite{PH}  that minimizers are locally in $C^{1,\beta-1}$ for $\beta=2/(2-p)$.

Furthermore, Phillips (cf. \cite{Phil1}) and Alt and Phillips (cf. \cite{AP}) showed that the free boundary is fully regular in dimension two. For the two-phase case, when $u$ is allowed to change signs, it was proved in \cite{GG} that $u$ is locally $C^{1,\beta-1}$. Moreover, the second author and Petrosyan proved in  \cite{LP}, the $C^1$ regularity of the free boundary in dimension two. However, none of these results say anything about the behavior near a fixed boundary, they are all interior results.

For the  particular case of the problem  when $p=0$, Alt, Caffarelli and Friedman introduced in \cite{ACF} a monotonicity formula and showed the optimal
 Lipschitz regularity of minimizers and the $C^1$ regularity of the free boundary in dimension two. In the case $p=1$, equation
 (\ref{eq:problem}) reduces to the two-phase obstacle problem which was introduced by Weiss in \cite{WE2}. For this problem, Ural'tseva and Shahgholian proved in \cite{UR} and \cite{Sha03} the optimal $C^{1,1}$ regularity. Furthermore, in \cite{SUW}, Shahgholian, Ural'tseva and Weiss proved
the $C^1$ regularity of the free boundary close to so called \textit{branching points} (see Section 2). The mentioned results are all interior regularity results. But for the cases $p=0$ and $p=1$ there are also some results concerning the behavior of the free boundary near the fixed boundary. See for instance \cite{AMM} and \cite{KKS} where it is proved for $p=1$ and $p=0$ respectively, that the free boundary approaches the fixed one in a tangential fashion.

\subsection{Organization of the paper}
The paper is organized as follows:
\begin{itemize}
\item In Section 2, we shall introduce the notion of blow-ups and also the different notions of free boundary points.
\item In Section 3, we prove $C^{1,\alpha}$-estimates up to the fixed boundary.
 \item In Section 4, we state and prove some technicalities that are important for the rest of the paper, such as growth estimates,
non-degeneracy, classification of global minimizers and Weiss's monotonicity formula.
\item In Section 5, we prove the main result.
\end{itemize}

\section{Free boundary points and the notion of blow-ups}
\renewcommand{\theequation}{2.\arabic{equation}}\setcounter{equation}{0}
 Suppose that $u$ is a minimizer of (\ref{eq:minimizer}) and $x_0\in \Gamma$. Then we divide the free
boundary points into the following parts (see Figure \ref{fig:f2}):
\begin{figure}[!h]
\begin{center}
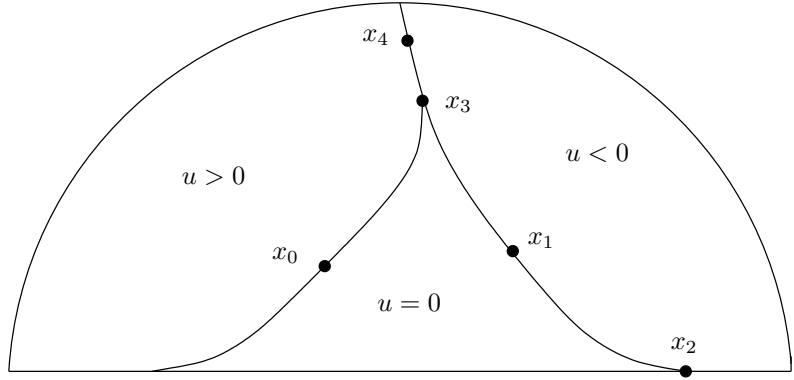
    \caption{This figure illustrates the different types of the free boundary points.
The point $x_0$ is a positive one-phase free boundary point, $x_1$  is a negative one-phase point, $x_2$ is a negative one-phase point touching the fixed boundary, $x_3$ is a branching point and $x_4$ is a two-phase point which might or might not be a branching point.
}    \label{fig:f2}
\end{center}
\end{figure}
\begin{enumerate}
 \item We say that $x_0$ a positive (negative) one-phase free boundary point if there exist a neighborhood of
 $x_0$ such that $u$ is non-negative (non-positive) in it. In other words,
$x_0\in \Gamma^+\setminus \Gamma^-$($x_0\in \Gamma^-\setminus \Gamma^+$).
\item We say that $x_0$ is a two-phase free boundary point if $x_0\in \Gamma^+\cap \Gamma^-$.
 Moreover, if $|\nabla u(x_0)|=0$ then $x_0$ is said to be a branching point.
 \end{enumerate}

A useful notion when studying properties of free boundary problems is the so-called blow-ups.

\begin{definition} For a given  minimizer  $u$ of (\ref{eq:minimizer}), $x_0\in \Gamma$ (one phase or branching point) we define the rescaled functions
\beas
u_{x_0,r}(x)=\frac{u(x_0+rx)}{r^\beta},\quad \beta=\frac{2}{2-p},\quad r>0.
\eeas
In the case $x_0=0$ we use the notation $u_r=u_{0,r}$. If we can find a sequence $u_{x_0,r_j},\,r_j\rightarrow 0$ such that
$$
u_{x_0,r_j}\rightarrow u_0\quad \text{ in } \, C^1_{\text{loc}}(\R^n\cap \{x_1>0\}) \textup{ (or $C^1_{\text{loc}}((\R^n))$)},
$$
we say that $u_0$ is a \textit{blow-up} of $u$ at $x_0$. It is easy to see that $u_0$ is a global minimizer of (\ref{eq:minimizer}), i.e., a minimizer in $\R^n\cap \{x_1>0\}$ or in $\R^n$, with a certain growth condition (see below).
\end{definition}

We also define the following class of global minimizers:
\begin{definition} Let $M$ be a positive constant. We define $P_\infty(M)$ to be the class of minimizers $u$ of
 (\ref{eq:minimizer}) in $\R^n\cap \{x_1>0\}$ such that $0\in \Gamma\cap \Pi$ and
  \beas
\|u\|_{L^{\infty}(B_R^+)}\leq MR^\beta,
\eeas
for all $R>0$.
 \end{definition}

\section{Regularity}
\renewcommand{\theequation}{3.\arabic{equation}}\setcounter{equation}{0}

\renewcommand{\theequation}{3.\arabic{equation}} \setcounter{equation}{0}
In this section we will prove that any minimizer is $C^{1,\alpha}$ up to the fixed boundary. It is possible that parts of the results in this section can be found in the literature, however we have not been able to find any good reference for that. For instance, in \cite{GG} the interior $C^1$-regularity is proved for minimizers of functionals of the type \eqref{eq:minimizer}, but nowhere can any statement about the regularity up to the fixed boundary be found, even though the technique properly used, probably would imply the same regularity up to the boundary in this case.
\begin{lemma}\label{lem:linfty}
 (Estimates in $L^\infty$) Let $u$ be a minimizer of (\ref{eq:minimizer}). Then $u\in L^\infty(B_1^+)$ and we have the estimate
 $$
 \|u\|_{L^\infty(B_1^+)}\leq C(p,\|f\|_{L^\infty}).
 $$
\end{lemma}

\begin{proof} Any minimizer of \eqref{eq:minimizer} is a solution of (\ref{eq:problem}) when $\{u\neq 0\}$.
 Let $$v(x)=\max (u(x),1).$$ Then $\Delta v\geq -pC,$  for some positive constant $C$. By the maximum principle
\beas
\sup v\leq \max (1,\sup f)+C.
\eeas
Similar arguments for 
$$v(x)=\max (-u(x),1),$$
 show that $u$ is
 bounded from below and we will get 
 $$\|u\|_{L^\infty}\leq C(pC+\|f\|_{L^\infty}).$$
\end{proof}
\subsection{H\"older Regularity}
We can now prove that minimizers are H\"older continuous for all exponents less than one.  Throughout the rest of the paper, the harmonic replacement of a function $u$ in an open set $D$, will refer to the function $v$ satisfying
$$
\left\{\begin{array}{ll} \lap v = 0 &\textup{in $D$,}\\ u =v &\textup {on $\partial D$.}
\end{array}\right.
$$

\begin{proposition} (H\"older regularity) Let $u$ be a minimizer of \eqref{eq:minimizer}. Then for each $\gamma<1$ there is a constant $C=C(\gamma,\lambda^\pm,u,p,\|u\|_{L^\infty(B^+_1)})$ such that
$$
\|u\|_{C^{0,\beta}(\overline{B_\frac12^+})}\leq C.
$$
\end{proposition}
\begin{proof} Take $x_0\in B_\frac12^+$ and let $0<r<\frac12$. The idea is to prove that for all $\gamma\in (0,1)$ there is a constant $C_\gamma$ independent of $r$ and $x_0$ such that
\be\label{eq:mor2}
\int_{B_r^+(x_0)}|\nabla u|^2\d x\leq C_\gamma r^{n-\gamma}.
\ee
By Morrey's embedding this will imply the desired result, see Theorem 7.19 in \cite{gt}. With $v$ as the harmonic replacement of $u$ in $B_r^+(x_0)$ we have, due to the Dirichlet principle,
$$
\int_{B_r^+(x_0)}|\nabla v|^2\d x\leq \int_{B_r^+(x_0)}|\nabla u|^2\d x\leq E(u).
$$
Since $v$ is harmonic and $u=v$ on $\partial B_r^+(x_0)$
$$
\int_{B_r^+(x_0)}|\nabla v-\nabla u|^2\d x=\int_{B_r^+(x_0)}|\nabla v|^2-|\nabla u|^2\d x.
$$
Putting these to together and using Lemma \ref{lem:linfty} we can conclude
\begin{align*}
\int_{B_r^+(x_0)}|\nabla v-\nabla u|^2\d x&\leq \int_{B_r^+(x_0)}2(\lambda_1(u^+)^p+\lambda_2 (u^-)^p)\d x\\
&\leq C(p,\lambda^\pm,\|f\|_{L^\infty(\partial B_1^+)})r^n.
\end{align*}

If $r<R<\frac12$ and $v$ is the harmonic replacement in $B_R^+(x_0)$ the estimate above implies via Young's inequality
\begin{align*}
\int_{B_r^+(x_0)}|\nabla u|^2\d x&\leq 2\int_{B_r^+(x_0)}|\nabla u-\nabla v|^2\d x+2\int_{B_r^+(x_0)}|\nabla v|^2\d x\\
&\leq 2\int_{B_r^+(x_0)}|\nabla u-\nabla v|^2\d x+2\int_{B_r^+(x_0)}|\nabla v|^2\d x\\
&\leq Cr^n+2C\left(\frac{r}{R}\right)^n\int_{B_R^+(x_0)}|\nabla v|^2 \d x \\
&\leq Cr^n+2C\left(\frac{r}{R}\right)^n\int_{B_R^+(x_0)}|\nabla u|^2 \d x,
\end{align*}
where we have again used that $v$ minimizes the Dirichlet energy and the estimate
$$
\int_{B_r^+(x_0)}|\nabla v|^2\d x\leq C\left(\frac{r}{R}\right)^n\int_{B_R^+(x_0)}|\nabla v|^2\d x,
$$
which follows from interior gradient estimates for harmonic functions, upon reflecting $v$ in an odd manner across $\Pi$. Taking $r=\sigma^{j+1}$ and $R=\sigma^j$ { where $\sigma $ is small enough and $j\in \N$ then} this turns into
$$
\int_{B_{\sigma^{j+1}}^+(x_0)}|\nabla u|^2\d x\leq C\sigma^{(j+1)n}+C\sigma^n\int_{B_{\sigma^j}^+(x_0)}|\nabla u|^2 \d x.
$$
Now it is clear that if \eqref{eq:mor2} holds for $r=\sigma^j$ for some $\gamma$ and $C_\gamma$, then the estimate above implies
\begin{align*}
\int_{B_{\sigma^{j+1}}^+(x_0)}|\nabla u|^2\d x&\leq C\sigma^{(j+1)n}+CC_\gamma\sigma^n\sigma^{j(n-\gamma)}\\
&\leq C_\gamma\sigma^{(j+1)(n-\gamma)}\left(\frac{C}{C_\gamma}+C\sigma^\gamma\right).
\end{align*}
If we choose $C_\gamma$ large enough and $\sigma$ small enough then
$$
\int_{B_{\sigma^{j+1}}^+(x_0)}|\nabla u|^2\d x \leq C_\gamma\sigma^{(j+1)(n-\gamma)}.
$$
Iterating this, yields \eqref{eq:mor2}.
\end{proof}
\subsection{$C^{1,\alpha}$-estimates up to the fixed boundary}
Now we turn our attention to the $C^{1,\alpha}$-regularity. The idea is to use the method in \cite{LS}. In what follows we will use the notation $B_r^+(x)=B_r(x)\cap \{x_1>0\}$.

We are going to employ the following result, which is a special case of Theorem I.2 in \cite{Cam}.

\begin{proposition}\label{prop:campanato} Let $u\in H^1(B_1^+)$. Assume there exist $C$, $\alpha$ such that for each $x_0\in B_\frac12^+$ there is a vector $A(x_0)$ with the property
\begin{equation}\label{eq:campanato}
\int_{B_r(x_0)\cap B_1^+}|\nabla u(x)-A(x_0)|^2\ d x\leq Cr^{n+2\alpha}{\text{ for every } r< \frac{1}{2}}.
\end{equation}
Then $u\in C^{1,\alpha}(\overline{B_\frac12^+})$ and we have the estimate
$$
\|u\|_{C^{1,\alpha}(\overline{B_\frac12^+})}\leq C_0(C).
$$
\end{proposition}
The only non-standard in the proposition above is that we get $C^{1,\alpha}$-estimates up to the fixed boundary. Below we present a technical result concerning harmonic functions. First we just make the following remark.

\begin{remark} Let $x_0\in B_\frac12^+$. Then for any $r<\frac12$, we have the following estimates for any harmonic function $u$ in $B_r^+(x_0)$ which up to a linear part vanishes on $B_1\cap \Pi$:
\begin{equation}\label{eq:submean}
 \sup_{B_\frac{r}2^+(x_0)}|D^2u(x)|\leq \frac{C}{r^{n/2+1}}\left(\int_{B_r^+(x_0)}|\nabla u|^2\d x\right)^\frac12,
 \end{equation}
and
\begin{equation}\label{eq:gradest} |\nabla u(x_0)|\leq  \frac{C}{r^{n/2}}\left(\int_{B_r^+(x_0)}|\nabla u|^2\d x\right)^\frac12 .
\end{equation}
Moreover, for $\alpha\in [0,1)$ there holds
\begin{equation}\label{eq:c1gamma}
\|u\|_{C^{1,\alpha}(B_\frac{r}{2}^+(x_0))}\leq Cr^{1-\alpha}||\lap u||_{B_r^+(x_0)}+Cr^{-n/2-\alpha}\left(\int_{B_r^+(x_0)}|\nabla u|^2\d x\right)^\frac12.
\end{equation}
To obtain these estimates, assume $r=1$ and simply  reflect $u$ (except its linear part)  oddly  across $\Pi$. Then we can apply usual interior estimates in $B_{r/2}^+(x_0)\cup (B_{r/2}^+(x_0))^\textup{reflected}$. In particular, the estimate \eqref{eq:submean} will now  follow from rescaling the estimate
$$
 \sup_{B_\frac{1}2^+(x_0)}|D^2u(x)|\leq C \left(\int_{B_1^+(x_0)} u^2\d x\right)^\frac12\leq C\left(\int_{B_1^+(x_0)}|\nabla u|^2\d x\right)^\frac12,
$$
where the first estimate comes from interior $C^2$-estimates for harmonic functions (see Theorem 7 on page 29 in \cite{E}). Similarly, \eqref{eq:gradest} follows from rescaling the gradient estimate for harmonic functions
$$
|\nabla u(x_0)|\leq C \left(\int_{B_r^+(x_0)}u^2\d x\right)^\frac12 \leq C \left(\int_{B_r^+(x_0)}|\nabla u|^2\d x\right)^\frac12.
$$
Finally, \eqref{eq:c1gamma} is a consequence of interior $C^{1,\alpha}$-estimates for the Poisson equation (cf. Theorem 4.15 on page 68 in \cite{gt})
\begin{align*}
\|u\|_{C^{1,\alpha}(B_\frac{1}{2}^+(x_0))}&\leq C||\lap u||_{B_1^+(x_0)}+C\left(\int_{B_1^+(x_0)} u^2\d x\right)^\frac12\\
&\leq C||\lap u||_{B_1^+(x_0)}+C\left(\int_{B_1^+(x_0)}|\nabla u|^2\d x\right)^\frac12.
\end{align*}
\end{remark}
\begin{lemma}\label{lem:sigmatheta} Let $x_0\in B_{1/2}^+$ and $v$ be harmonic in $B_r^+(x_0)$ and assume also that $v$ vanishes on $B_1\cap \Pi$  up to a linear part. Then for $\sigma<1$ there holds
$$
\int_{B_{\sigma r}^+(x_0)}|\nabla v(x)-\nabla v(x_0)|^2\d x\leq C\sigma^{n+2}\int_{B_r^+(x_0)}|\nabla v(x)|^2\d x.
$$
\end{lemma}
\begin{proof} From the estimates for the second derivatives for harmonic functions we have
$$
\sup_{B_{\sigma r}^+(x_0)}|D^2v(x)|\leq \frac{C}{r^{n/2+1}}\left(\int_{B_r^+(x_0)}|\nabla v|^2\d x\right)^\frac12,
$$
from which it follows that for $x\in B_{\sigma r}^+(x_0)$
$$
|\nabla v(x)-\nabla v(x_0)|^2\leq \frac{C\sigma r}{r^{n+2}}\int_{B_r^+(x_0)}|\nabla v|^2\d x.
$$
If we integrate this over $B_{\sigma r}^+(x_0)$ we obtain
\begin{align*}
\int_{B_{\sigma r}^+(x_0)}|\nabla v(x)-\nabla v(x_0)|^2\d x\leq C\sigma^{n+2}\int_{B_r^+(x_0)}|\nabla v|^2\d x.
\end{align*}
\end{proof}

Now we are ready to prove the desired estimate.

\begin{proposition}($C^{1,\alpha}$-estimates)\label{prop:c1a} Let $u$ be a minimizer of \eqref{eq:minimizer}. Then there are constants $\alpha=\alpha(\lambda^\pm,u,p,\|u\|_{L^\infty(B_1)})$ and $C=C(\lambda^\pm,u,p,\|u\|_{L^\infty(B_1)})$ such that
$$
\|u\|_{C^{1,\alpha}(\overline{B_\frac12^+})}\leq C.
$$
\end{proposition}
\begin{proof} We will  find  appropriate constants $\alpha$ and $C$ such that \eqref{eq:campanato} holds for all $r<\frac{1}{2}$. Then the result will follow from Proposition \ref{prop:campanato}.

The way we will do this is by proving that for some small $\alpha,\,\sigma$ and for all $x_0\in B_\frac12^+$ we can find a sequence $A_j$ such that
\begin{equation}\label{eq:jest1}
\int_{B_{\sigma^j}^+(x_0)}|\nabla u-A_j|^2\d x\leq C_1 \sigma^{j(n+2\alpha)},
\end{equation}
and
\begin{equation}\label{eq:jest2}
|A_j-A_{j-1}|\leq C_2\sigma^{j\alpha},
\end{equation}
for all $j$, as long as we have
\begin{equation}\label{eq:jest3}
\inf_{B_{\sigma^j}^+(x_0)}|u|\leq \sigma^j.
\end{equation}
Intuitively this will imply the desired inequality since if \eqref{eq:jest3} holds for all $j$ then we can pass to the limit in \eqref{eq:jest1} and we are done, if not, \eqref{eq:jest3} must fail for some $j$, but then $u$ does not vanish in the corresponding ball so that the equation for $u$ is non-singular there, and we can use estimates for the Poisson equation with bounded inhomogeneity.

For the sake of clarity we split the proof into three different steps.

\noindent {\bf Step 1: \eqref{eq:jest1} holds as long as \eqref{eq:jest3} holds.} The proof is by induction. Clearly, this is true for $j=1$ and some $A_1$ if we pick $C_1$ large enough. So assume this is true for $j=k$ and then we prove that it holds also for $j=k+1$. Take $v$ to be the harmonic replacement of $u$ in $B_{\sigma^k}^+(x_0)$. Then $v-A_k\cdot x$ is the replacement of $u-A_k\cdot x$. Hence, by the Dirichlet principle,
$$
\int_{B_{\sigma^k}^+(x_0)}|\nabla v-A_k|^2\d x\leq \int_{B_{\sigma^k}^+(x_0)}|\nabla u-A_k|^2\d x=I_1.
$$
Let $A_{k+1}=\nabla v(x_0)$. Lemma \ref{lem:sigmatheta} implies
$$
\int_{B_{\sigma^{k+1}}^+(x_0)}|\nabla v-A_{k+1}|^2\d x\leq C\sigma^{n+2}\int_{B_{\sigma^{k}}^+(x_0)}|\nabla v-A_{k}|^2\d x\leq C\sigma^{n+2} I_1.
$$
Since $u$ is a minimizer of \eqref{eq:minimizer}, we have
$$
\int_{B_{\sigma^k}^+(x_0)}|\nabla v|^2\d x\leq \int_{B_{\sigma^k}^+(x_0)}|\nabla u|^2\d x\leq \int_{B_{\sigma^k}^+(x_0)}|\nabla u|^2+\lambda_1(u^+)^p+\lambda_2 (u^-)^p\d x.
$$
Using that \eqref{eq:jest3} is assumed to hold up to $j=k$, the H\"older regularity of $u$ implies
\begin{align*}
I_2&=\int_{B_{\sigma^k}^+(x_0)}|\nabla u-\nabla v|^2\d x\leq \int_{B_{\sigma^k}^+(x_0)}\lambda_1(u^+)^p+\lambda_2(u^-)^p\d x\\&\leq \max(\lambda_i)\sigma^{kn}\sup_{B_{\sigma^k}^+(x_0)}|u|^p\leq C\max(\lambda_i)\sigma^{k(n+\beta p)}.
\end{align*}
Now pick $\beta$ so that $\beta p>2\alpha$. By Young's inequality
\begin{align*}
\int_{B_{\sigma^{k+1}}^+(x_0)}|\nabla u-A_{k+1}|^2\d x&\leq 2\int_{B_{\sigma^{k+1}}^+(x_0)}|\nabla v-A_{k+1}|^2\d x+\\
 &\qquad +2\int_{B_{\sigma^{k+1}}^+(x_0)}|\nabla u-\nabla v|^2\d x\\
&\leq 2 C \sigma^{n+2} I_1+2  C\sigma^{k(n+\beta p)}\\
&\leq 2 C_1 C\sigma^{n+2}\sigma^{k(n+2\alpha)}+2C\sigma^{k(n+\beta p)}\\
&\leq \sigma^{(k+1)(n+2\alpha)}\left(C_1C\sigma^{2-2\alpha}+2\frac{C}{C_1}\sigma^{\beta p k-n-2\alpha(k+1)}\right)\\
&\leq \sigma^{(k+1)(n+2\alpha)}\left(C_1+2C\sigma^{-n-2\alpha}\right)\\
&\leq C_1\sigma^{(k+1)(n+2\alpha)},
\end{align*}
if $C_1$ is chosen to be large enough and $\sigma$ small enough. This proves that \eqref{eq:jest1} 
  holds for $j=k+1$.\\

\noindent {\bf Step 2: \eqref{eq:jest2} holds as long as \eqref{eq:jest3} holds.}  We remark that  $A_{k+1}-A_k$ is the gradient of $v-A_k\cdot x$ at $x_0$, where $v$ is as in Step 1. Therefore, by the $C^1$-estimates in \eqref{eq:gradest} there holds
$$
|A_{k+1}-A_k|\leq C\sigma^{-kn/2}\left(\int_{B_{\sigma^k}^+(x_0)}|\nabla v-A_k|^2\d x\right)^\frac12\leq C\sqrt{C_1}\sigma^{\alpha k},
$$
from \eqref{eq:jest1} for $j=k$, which holds due to Step 1. Hence, if $C_2$ is large enough, $|A_{k+1}-A_k|\leq C_2\sigma^{\alpha (k+1)}$.\\

\noindent {\bf Step 3: Conclusion.} First of all, in the case when \eqref{eq:jest3} 
  holds for all $j$ then from \eqref{eq:jest2}
$$
|A_j-A_k|\leq \sum_{i=j}^{k-1}|A_i-A_{i+1}|\leq C'\sigma^{\alpha j},
$$
 Hence, the sequence $A_j$ converges to a limit $A(x_0)$. This together with \eqref{eq:jest1} implies \eqref{eq:campanato} immediately.

If \eqref{eq:jest3} holds for $j<k$ but fails for $j=k$ then
$$
\inf_{B^+_{\sigma^k}(x_0)}|u|>\sigma^k,
$$
so that from \eqref{eq:problem} we have
$$
|\lap u|\leq C(p,\lambda^\pm)\sigma^{k(p-1)}\textup{ in $B^+_{\sigma^k}(x_0)$.}
$$
Hence, $u-A_k\cdot x$ has $C^{1,\alpha}$-estimates in $B_{\sigma^k/2}^+(x_0)$. In particular from \eqref{eq:c1gamma} we have

\begin{align*}
|\nabla u(x_0)-A_k|&\leq C(p,\lambda^\pm)\sigma^{kp}+C\sigma^{-kn/2}\left(\int_{B^+_{\sigma^k}(x_0)}|\nabla u-A_k|^2\d x\right)^\frac12\\
&\leq C(p,\lambda^\pm)\sigma^{kp}+C\sqrt{C_1}\sigma^{k\alpha}\\
&\leq \sigma^{k\alpha}\left(C(p,\lambda^\pm)+C\sqrt{C_1}\right)\leq C\sigma^{k\alpha},
\end{align*}
if $p\geq\alpha$, and also from \eqref{eq:c1gamma} it follows that from $r\leq \sigma^k$
\begin{align*}
r^{-\alpha}\displaystyle\operatorname{osc}_{B^+_{r/2}(x_0)}|\nabla u-A_k|&\leq C(p,\lambda^\pm)r^{(p-\alpha)}+Cr^{-n/2-\alpha }\left(\int_{B^+_{r}(x_0)}|\nabla u-A_k|^2\d x\right)^\frac12\\
&\leq C(p,\lambda^\pm)r^{(p-\alpha)}+C\sqrt{C_1}r^{(\alpha-\alpha)}\\
&\leq  C,
\end{align*}
if again $p\geq\alpha$. With $A(x_0)=\nabla u(x_0)$ and $\sigma\leq 1/2$, integrating the last two estimates over $B_r^+(x_0)$ yields for any $r\leq \sigma^{k+1}$
$$
\int_{B^+_r(x_0)}|\nabla u-A(x_0)|^2\d x\leq Cr^{n+2\alpha}.
$$
For $r= \sigma^j$ for $j\leq k$ we have from Young's inequality and \eqref{eq:jest1}
\begin{align*}
\int_{B_r^+(x_0)}|\nabla u-A(x_0)|^2\ d x&\leq 2\int_{B_r^+(x_0)}|\nabla u-A_j|^2\d x+2\int_{B_r^+(x_0)}|A(x_0)-A_j|^2\d x\\
&\leq 2 C_1\sigma^{n+j 2\alpha}+2\sigma^n |A(x_0)-A_j|^2.
\end{align*}
From \eqref{eq:jest2} for $j\leq k$ it follows that
$$
|A(x_0)-A_j|\leq |A(x_0)-A_k+A_k-A_{k-1}+\cdots+A_{j+1}-A_j|\leq C\sigma^{j\alpha}.
$$
This yields the estimate, still with $r=\sigma^j$, for $j\leq k$
$$
\int_{B_r^+(x_0)}|\nabla u-A(x_0)|^2\ d x\leq 2 C_1\sigma^{n+j 2\alpha}+2C\sigma^{n+j 2\alpha},
$$
thus, we obtain the desired inequality for all $r$.
\end{proof}

\section{Optimal growth}
\renewcommand{\theequation}{4.\arabic{equation}}\setcounter{equation}{0}

In the proof the proposition below, we will use techniques similar to those in for instance \cite{AMM} and \cite{CKS} to prove that $u$ will have the optimal growth of order $\beta=2/(2-p)$ at branching points.

\begin{proposition}\label{prop:growth}
 (Optimal growth) Suppose  $u\in P_1(M)$, $x_0\in \Gamma\cap \Pi$ and $|\nabla u(x_0)|=0$.
 Then there exists a constant $C=C(M)$ such that with $\beta=\frac{2}{2-p}$
$$
\sup_{B_r^+(x_0)}|u|\leq Cr^\beta,\quad \text{for all }\,\, 0<r<\frac{1}{2}.
$$
\end{proposition}

\begin{proof} The proof is by contradiction. Without loss of generality, assume $x_0=0$ and define $$S_r(u)=\sup_{ B_r^+} u,$$  for $ 0<r<\frac{1}{2}$. We will show that either $S_r\leq C r^\beta$ for a constant $C$ or there exists a $k\in \N$ with $2^kr\leq 1$ such that $S_r\leq 2^{-k\beta}S_{2^kr}$. Suppose both these assertions  fail, then one can find sequences $r_j\to 0$, $u_j\in P_1(M)$ such that with $S_j:=S_{r_j}$ there holds
\beas\label{eq:non1}
S_j > C_jr_j^\beta,
\eeas
where $C_j \rightarrow \infty$  and
\beas\label{eq:non2}
 S_j > 2^{-k\beta}S_{r_j2^k},\quad \text{for all  } k\in\N,\text{ and } 2^{k}r_j\leq 1.
\eeas
Define \beas w_j(x)=\frac{u_j(r_j x)}{S_j}.\eeas Then:
\begin{itemize}
  \item [(a)] $\sup_{B_1^+}|w_j(x)|=1,$
  \item[(b)]  $\sup_{B_{2^k}^+}|w_j(x)|\leq 2^{k\beta},$
\item[(c)] $w_j(0)=|\nabla w_j(0)|=0,$
\item [(d)] $w_j=0$ on $B_\frac{1}{r_j}\cap \Pi$,
\item[(e)] $w_j$ is a minimizer of
\beas
\int_{B_{2^j}^+} \bigg(\frac{|\nabla v|^2}{2}+T_j\big( \lambda^+(v^+)^{p}+2\lambda^-(v^-)^{p}\big)\bigg),
\eeas
where $T_j=\frac{r_j^{-2}}{S_j^{2-p}}\to 0$ as $j\rightarrow \infty.$
\end{itemize}
By using Proposition \ref{prop:c1a}, we can find a subsequence of $w_j$ which converges to a limiting function $w_0$ in $C^1(\overline{B_R^+})$ for all $R>0$. Due to (a)-(e), $w$ satisfies
\begin{enumerate}
 \item $\sup_{B_1^+}|w_0(x)|=1$,
\item $\sup_{B^+_{2^k}}|w_0(x)|\leq 2^{k\beta}$ for all $k$,
\item $w_0(0)=|\nabla w(0)|=0$,
\item $w_0=0$ on $\Pi$,
\item $\Delta w_0=0$ in $(\R^n)^+$.
\end{enumerate}
We  reflect the function $w_0$ in an odd manner with respect to $\Pi$ to get a harmonic function in the whole $\R^n$.
 By interior estimates for harmonic functions and (2), for every $k\geq 1$ we have
\beas
\sup_{B_{2^k}}|D^2w_0(z)|\leq \frac{C}{2^{k(n+2)}}\|w_0\|_{L^1(B_{2^k})}\leq C 2^{k(\beta -2)}.
\eeas
Since $\beta < 2$, passing $k\to \infty$ implies $D^ 2 w_0 =0$ and consequently
$w_0$ is a linear function. Then (3) implies $w_0=0$, contradicting (1).
\end{proof}
\section{Technical tools}
Here we present some technical lemmas which we will use later to prove our main result.
\subsection{Non-degeneracy}
The next lemma shows that blow-ups cannot vanish identically. This property is usually referred to as non-degeneracy and to prove it,
 we use the idea in \cite{LP} which in turn is an adaptation of a similar proof given in \cite{CR}.
\begin{lemma}\label{lem:degen}
(Non-degeneracy) Suppose that $u$ is a minimizer of \eqref{eq:minimizer} and $x_0\in\Gamma^ + \cap \Pi$.
Then for some constant $c^ +=c^+(\lambda^+)$
\be\label{eq:nonn1}
\sup_{\pa B_r^+(x_0)\cap \O^+}u\geq c^+r^\beta,\quad 0<r<\frac{1}{2}.
\ee
Similarly if $x_0 \in\Gamma^-\cap \Pi  $, then there exists a constant $c^ -=c^-(\lambda^-)$
\be\label{eq:nonn2}
\inf_{\pa B^ +_r(x_0)\cap \O^-}u\leq -c^-r^\beta, \quad 0<r<\frac{1}{2}.
\ee
\end{lemma}

\begin{proof} We prove only (\ref{eq:nonn1}). The inequality (\ref{eq:nonn2}) can derived analogously. Suppose that, $y\in \Omega^+$, $B_r^+(y)\subset B_1^+$ and $u$ is a minimizer of  (\ref{eq:minimizer}). Define the function
\beas
w(x)=|u(x)|^{\frac{2}{\beta}}-c|x-y|^2,
\eeas
where $c$ is a constant which we will determine later. By a simple computation we find
\beas
\Delta w=\frac{2p}{\beta}+\frac{2}{\beta}(\frac{2}{\beta}-1)\frac{|\nabla u|^2}{|u|^p}-4c, ,\quad \text{in } \Omega^+\cap B_r^+(y).
\eeas
If we choose $c=\frac{p}{2\beta}$ then $\Delta w\geq 0$ in $B_r^+(y)\cap \Omega^ +$ and by the maximum principle, the maximum of $w$ occurs on
$\pa (B_r^+(y)\cap \Omega^ +)$.
We know that
\beas
\begin{cases}
 w(y)\geq 0,\\
\Delta w \geq \text{in } B_r^+(y)\cap \Omega^ +,\\
w\leq 0\,\,\text{on } \pa \Omega^+,\\
w\leq 0\,\,\text{on } B_r^+(y)\cap \Pi,\\
\end{cases}
\eeas
and consequently $w$ attains its maximum on $\pa B_r^+(y)$ and $$\sup_{ \pa B_r^+(y)\cap \Omega^+} w >0.$$ In other words,
\be\label{eq:for y}
\sup_{ \pa B_r^+(y)\cap \Omega^+} u^{\frac{2}{\beta}}>cr^ 2.
\ee
Now let $x_0\in \Gamma^+\cap \Pi$. Then one can find a sequence ${y_j}$ in $\Omega^+$ such that $y_j\rightarrow x_0$.
Then by considering  (\ref{eq:for y}) for $y_j$ and passing to the limit, one obtains
\beas
\sup_{ \pa B_r^+(x_0)\cap \Omega^+} u^{\frac{2}{\beta}}\geq cr^ 2,
\eeas
or equivalently,
\beas
\sup_{\pa B_r^+(x_0)\cap \O^+}u\geq c^+r^\beta. 
\eeas
\end{proof}

One important consequence of Lemma \ref{lem:degen} is that the free boundary is stable in the sense that limits of free boundary points are are always free boundary points. In particular, it implies that if $u_j$ is a sequence of minimizers converging to $u_0$ and $x_j\in \Gamma^\pm(u_j)$ with $x_j\to x_0$, then $x_0\in \Gamma^\pm(u_0)$.

\subsection{Monotonicity formula}
The next lemma is a crucial monotonicity formula due to Weiss, proved in \cite{Wei04}. See Theorem 3.1 in  \cite{W}, where the monotonicity formula was introduced in the interior setting.

\begin{lemma}\label{lem:Weiss}(Weiss's monotonicity formula) Suppose that $u\in P_R(M)$,$0<r<R$  and $G(u)=2\lambda^+(u^+)^{p}+2\lambda^-(u^-)^{p} $. Let
\beas
W(r,x_0,u) = r^{-2\beta}\int_{ B^+_r(x_0)}\big(|\nabla u|^2+2G(u)\big)dx-\frac{\beta}{r^{1+2\beta}}\int_{\pa B^+_r(x_0)}u^2(x)\,ds,
\eeas
for $r>0$. Then $W$ is monotonically increasing with respect to $r$ if $r<d(\pa B_R^+,x_0)$. Moreover, $W$ is constant if and only if $u$ is a homogeneous function of degree $\beta$.

\end{lemma}

\subsection{Global minimizers}
The next theorem classifies the homogeneous global minimizers of (\ref{eq:minimizer}) in two dimensions. This result is basically a result from \cite{LP}. From this we can then classify all global minimizers. From now on we will be working only in two dimensions.

\begin{theorem}\label{thm:homo}
   Let $u\in P_\infty(M)$ be homogeneous and assume the dimension to be two. Then for some suitable constants $c^\pm$ one of the following holds:
\begin{enumerate}
 \item $u(x)=c^+(x_1^+)^\beta$, for one phase non-negative points,
\item $u(x)=-c^-(x_1^-)^\beta$, for one phase non-positive points.
\end{enumerate}
\end{theorem}

\begin{proof} Let $0\in \Gamma^+\cap \Pi$. Assume first that $u_0$ be a homogeneous global minimizer of \eqref{eq:minimizer}. From the homogeneity assumption, we conclude that any connected component of $\O^+$ is a cone. Lemma 4.2 in \cite{LP} asserts that it has opening  $\gamma \in (\pi/\beta,\pi)$, for $\beta=2/(2-p)$. Since $\beta\in (1,2)$, there can only be one component. Applying the second part of Lemma 4.2, we obtain $\gamma=\pi$, which up to rotations corresponds to $u_0(x)=c^+(x_1^+)^\beta$. Since $u_0$ must vanish on $\Pi$, no other rotation except the identity is possible. The case $0\in \Gamma^-\cap \Pi$ can be handled similarly.
\end{proof}

The theorem above implies in particular that there can be no two-phase points touching the fixed boundary.

\begin{corollary}\label{cor:onephase} Suppose $u\in P_1(M)$. Then the origin is a one-phase point.

\end{corollary}
\begin{proof} If there were to be a two-phase branching point touching $\Pi$, then we could by Proposition \ref{prop:growth} and the $C^1$-estimates perform a blow-up at the origin. Due to Lemma \ref{lem:degen}, the blow-up will have both phases non-empty, which by the theorem above is not possible. Now, if there is a two-phase point in $\Pi$ where the gradient does not vanish, then the gradient must be perpendicular to $\Pi$, which would imply that it is a one-phase point, a contradiction.
\end{proof}

\begin{lemma}\label{lem:globala} Suppose $u\geq 0$ is a minimizer of \eqref{eq:minimizer} in $\R^n\cap\{x_1>-A\}$ for some constant $A>0$ and that
$$
0\in \Gamma\cap \Pi,\quad \sup_{B_r}|u|\leq Cr^\beta,
$$
for $r>0$ and some $C>0$. Then $u$ is one of the alternatives in Theorem \ref{thm:homo}.
\end{lemma}
\begin{proof} We prove that $u$ is homogeneous of degree $\beta$ . Then  $u\in P_\infty(C)$ for some $C$ and the result follows from Theorem \ref{thm:homo}.

Since $u$ grows at most like $r^\beta$ at infinity,
$$
u_r(x)=\frac{u(rx)}{r^\beta}
$$
is bounded as $r\to \infty$. Using Proposition \ref{prop:growth}, the $C^1$-estimates and Lemma \ref{lem:degen}, we can extract a subsequence $u_j=u_{r_j}$, with $r_j\to\infty$
so that $u_j\to u_\infty$ where $u_\infty$ is a minimizer of \eqref{eq:minimizer} in $\R^n\cap\{x_1>0\}$, $u_\infty=0$ on $\{x_1=0\}$, $0\in \Gamma(u_\infty)$ and
$$
W(u_\infty,s)=\lim_{r\to \infty} W(u_r,s)=\lim_{r\to \infty} W(u,rs)=\lim_{r\to \infty} W(u,r).
$$
Then Lemma \ref{lem:Weiss} implies that $u_\infty$ is homogeneous of degree $\beta$ and $u\in P_\infty(C)$. From Theorem \ref{thm:homo}, we have $u_\infty=c^+(x_1^+)^\beta$.

We have also that $u_r$ is uniformly bounded when $r$ is small enough. Hence, by Proposition \ref{prop:growth}, the $C^1$-estimates and Lemma \ref{lem:degen}, we can extract a subsequence $u_{r_j}\to u_0$ for some subsequence $r_j\to 0$ such that $u_0$ is a minimizer of \eqref{eq:minimizer} in $\R^n$, $0\in \Gamma(u_\infty)$ and
$$
W(u_0,s)=\lim_{r\to 0} W(u_r,s)=\lim_{r\to 0} W(u,rs)=\lim_{r\to 0} W(u,r),
$$
which is a constant since $W$ is monotone. Hence, $W(u_0,s)$ is constant and then by Lemma \ref{lem:Weiss} $u_0$ must be homogeneous of degree $\beta$. Since $u\geq 0$, Theorem 4.1 in \cite{LP} implies that $u_0=u_\infty$.

Using Lemma \ref{lem:Weiss} again, it follows that
$$
W(u_0,1)\leq W(u,r)\leq W(u_\infty,1)=W(u_0,1),
$$
so that $W(u,r)$ is constant and $u$ must be homogeneous of degree $\beta$.
\end{proof}

\section{Proof of the main theorem}
\renewcommand{\theequation}{4.\arabic{equation}} \setcounter{equation}{0}

In this section we prove our main theorem.
In the proposition that follows we prove that near $\Pi$, the free boundary will have a normal very close to $e_1$ (see Figure \ref{fig:cones}), still in two dimensions. By Corollary \ref{cor:onephase}, any free boundary point touching $\Pi$ must be a one-phase point, hence we can work under the assumption that $u$ has a sign near the origin. In what follows, we will use the notation
$$
K_\delta(z)=\{|x_1-z_1|<\delta\sqrt{(x_2-z_2)^2+\cdots+(x_n-z_n)^2}\}.
$$
\begin{figure}[!h]
\begin{center}
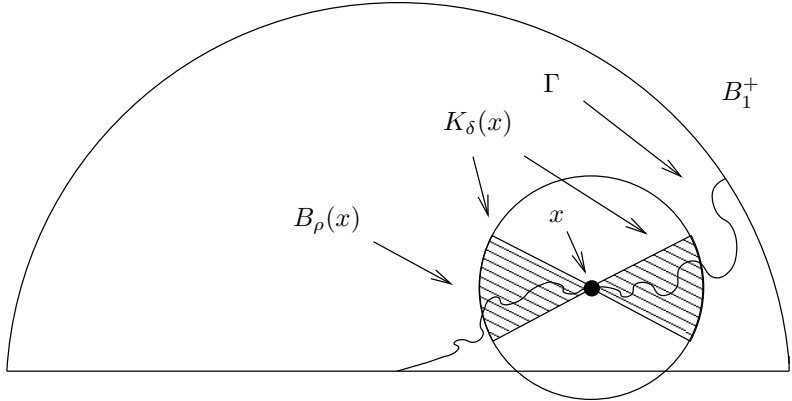
\caption{$\Gamma$ is inside $K_\delta(x)$ when $x$ is close to $\Pi$.}
\label{fig:cones}
\end{center}
\end{figure}
\begin{proposition}\label{prop:cone} Let $u\in P_1(M)$. For any $\delta>0$ there are $\e=\e(M,\delta)$ and $\rho=\rho(M,\delta)$ so that $x\in \Gamma$ and $x_1<\e$ imply
$$
\Gamma\cap B_\rho^+(x)\subset K_\delta(x)\cap B_\rho^+(x).
$$
\end{proposition}
\begin{proof} We argue by contradiction and we treat only the case when $u\geq 0$ near the origin. If the assertion is not true then for some $\delta>0$ there are sequences $u_j\in P_1(M)$, $\e_j\to 0$, $x^j\in \Gamma(u_j)$ and
$$
y^j\in \Gamma(u_j)\cap K_\delta^c(x^j).
$$
Let $r_j=|x^j-y^j|$. We split the proof into two different cases, depending on whether $y^j$ is very close to $x^j$ or not.

\noindent {\bf Case 1: $x^j_1/r_j$ bounded.} By choosing a subsequence we can assume $x^j/r_j\to A<\infty$. Let
$$
v_j(x)=\frac{u_j(r_jx+x^j)}{r_j^\beta}.
$$
Then $v_j$ satisfies:
\begin{enumerate}
\item From the optimal growth
$$
\sup_{B_R}|v_j|\leq CR^\beta
$$
for $Rr_j<1$.
\item $v_j$ is a minimizer of \eqref{eq:minimizer} in $B_{1/r_j}\cap \{x_1>-x^j/r_j\}$.
\item $v_j=0$ on $B_{1/r_j}\cap \{x_1=-x^j/r_j\}$.
\item $0\in \Gamma(v_j)$.
\item $z_j=(x^j-y^j)/r_j\in \dd B_1\cap K_\delta^c\cap \Gamma(v_j)$.
\item $v_j\geq 0$ in $B_R$ for $R$ small enough or $j$ large enough.
\end{enumerate}
Therefore, invoking Lemma \ref{lem:degen} and using the $C^1$-estimates for minimizers, we can assume that $v_j\to v_0$ locally uniformly and $z_j\to z_0$ such that:
\begin{enumerate}
\item $$
\sup_{B_R}|v_0|\leq CR^\beta,\textup{ for all $R>0$}.
$$
\item $v_0$ is a minimizer of \eqref{eq:minimizer} in $\R^n\cap \{x_1>-A\}$.
\item $v_0=0$ on $\R^n\cap \{x_1=-A\}$.
\item $0\in \Gamma(v_0)$.
\item $z_0\in \dd B_1\cap K_\delta^c\cap \Gamma(v_0)$.
\item $v_0\geq 0$.
\end{enumerate}
Lemma \ref{lem:globala} implies that $v_0=c^+(x_1^+)^\beta$. This contradicts (5).

\noindent {\bf Case 2: $x^j_1/r_j \to \infty$.} Define in this case
$$
v_j(x)=\frac{u_j(x^j_1x+x^j)}{(x^j_1)^\beta}.
$$
Then the following holds:
\begin{enumerate}
\item From the optimal growth
$$
\sup_{B_R}|v_j|\leq CR^\beta
$$
for $Rx^j_1<1$.
\item $v_j$ is a minimizer of \eqref{eq:minimizer} in $B_{1/x^j_1}\cap \{x_1>-1\}$.
\item $v_j=0$ on $B_{1/x^j_1}\cap \{x_1=-1\}$.
\item $0\in \Gamma(v_j)$.
\item $z_j=(x^j-y^j)/x^j_1\in \dd B_1\cap K_\delta^c\cap \Gamma(v_j)$.
\item $v_j\geq 0$ in $B_R$ for $R$ small enough or $j$ large enough.
\end{enumerate}
From the assumption on $x^j_1$ and $r_j$ it is clear that $z_j\to 0$. Moreover, from Theorem 8.2 in \cite{AP}, $\Gamma(v_j)$ is a uniform (in $j$) $C^1$-graph near the origin. Hence, (5) implies that $\Gamma(v_j)$ has asymptotically a tangent lying in $K_\delta^c$. Therefore, we can assume $v_j\to v_0$ locally uniformly where $v_0$ satisfies:
\begin{enumerate}
\item
$$
\sup_{B_R}|v_0|\leq CR^\beta\textup{ for all $R>0$.}
$$
\item $v_0$ is a minimizer of \eqref{eq:minimizer} in $\R^n\cap \{x_1>-1\}$.
\item $v_0=0$ on $\R^n\cap \{x_1=-1\}$.
\item $0\in \Gamma(v_0)$.
\item $\Gamma(v_0)$ has a tangent in $K_\delta^c$ at the origin.
\item $v_0\geq 0$.
\end{enumerate}
From Lemma \ref{lem:globala} we have $v_0=c^+(x_1^+)^\beta$. This is in contradiction with (5).
\end{proof}

Now the situation is as follows. Away from $\Pi$, Theorem 8.2 in \cite{AP} applies, so there the free boundary is a $C^1$-graph. Moreover, from Proposition \ref{prop:cone}, we know that the normal of the free boundary approaches $e_1$ as we approach $\Pi$. This is enough to assure that the free boundary is a uniform $C^1$-graph up to $\Pi$. We spell out the details below.

\begin{proof}[Proof of Theorem \ref{thm:main}] Since any free boundary point in $\Pi$ must be a one-phase point, we can assume $0\in \Gamma^+\cap \Pi$. Denote by $\nu_x$ the normal of $\Gamma$ at a point $x$. We need to prove that $\nu_x$ is uniformly continuous. From Theorem 8.2 in \cite{AP} it follows that $\Gamma$ is a $C^1$-graph away from $\Pi$. In particular, around any point $x\in \Gamma$, $\nu$ is continuous with a modulus of continuity $\sigma(\cdot/x_1)$, where $\sigma$ is some modulus of continuity. Moreover, by Proposition \ref{prop:cone}, we know that for any $\tau>0$, there is a $\delta_\tau$ such that $x_1<\delta_\tau$ implies $\|\nu_x-e_1\|<\tau/2$.

Take two points $x,y\in \Gamma$. Now we split the proof into three cases:\\
\noindent {\bf Case 1:} $x_1,y_1<\delta_\tau/2$. Then obviously $\|\nu_x-\nu_y\|\leq \tau$.

\noindent{\bf Case 2:} $x_1<\delta_\tau/2$ and $y_1>\delta_\tau/2$. Then $|x-y|<\delta_\tau/2$ implies $\|\nu_x-\nu_y\|\leq \tau$.

\noindent{\bf Case 3:} $x_1,y_1>\delta_\tau/2$. From the arguments above,
$$
\|\nu_x-\nu_y\|\leq \sigma(2|x-y|/\delta_\tau),
$$
which implies that $\|\nu_x-\nu_y\|\leq \tau$ if $|x-y|$ is small enough.

Hence we have proved that $\nu_x$ is uniformly continuous.
\end{proof}
\section{Acknowledgments} Part of this work was carried out when the authors were at MSRI,
Berkeley, during the program ''Free boundary problems''. We thank everyone at this institute for their great hospitality. The second author was at
 this time also employed by MSRI. Moreover, the first author was partially supported by Stockholm University.
 \bibliographystyle{amsplain}
\bibliography{ref.bib}

 \end{document}